\documentclass[11pt, reqno]{amsart}

\usepackage[margin=1.3in]{geometry}
\usepackage{amsmath, amssymb, amsthm}
\usepackage{amsfonts, amscd, bm, cancel}
\usepackage{amsrefs}
\usepackage{xcolor}
\usepackage[colorlinks=true, urlcolor=blue, linkcolor=blue, citecolor=magenta, pdfstartview=FitH]{hyperref}
\newtheorem{theorem}{Theorem}[section]
\newtheorem*{thmA}{Theorem A}
\newtheorem{lemma}[theorem]{Lemma}

\theoremstyle{remark}
\newtheorem{remark}[theorem]{Remark}
\newtheorem{example}{Example}
\newtheorem*{ack}{Acknowledgments}

\numberwithin{equation}{section}
\allowdisplaybreaks
\begin{document}

\title[Alexandrov-Fenchel type inequalities]{Alexandrov-Fenchel type inequalities with convex weight in space forms}

\author[K.-K. Kwong]{Kwok-Kun Kwong}
\address{University of Wollongong\\
Northfields Ave\\
2522 NSW, Australia}
\email{\href{mailto:kwongk@uow.edu.au}{kwongk@uow.edu.au}}
\author[Y. Wei]{Yong Wei}
\address{School of Mathematical Sciences, University of Science and Technology of China, Hefei 230026, P.R. China}
\email{\href{mailto:yongwei@ustc.edu.cn}{yongwei@ustc.edu.cn}}
\date{\today}

\begin{abstract}
In this paper, we derive new sharp weighted Alexandrov-Fenchel and Minkowski inequalities for smooth, closed hypersurfaces under various convexity assumptions in Euclidean, spherical, and hyperbolic spaces. These inequalities extend classical results by incorporating weights given by convex, non-decreasing positive functions, which are otherwise arbitrary. Our approach gives rise to a broad family of geometric inequalities, as each convex, non-decreasing function yields a corresponding inequality, providing considerable flexibility.
\end{abstract}

\maketitle

\tableofcontents

\section{Introduction}

The Alexandrov-Fenchel and Minkowski inequalities are fundamental in differential geometry, geometric analysis, and convex geometry. These inequalities establish fundamental relationships between volume, surface area, and curvature integrals of hypersurfaces enclosing bounded domains. These inequalities have important applications ranging from isoperimetric problems to eigenvalue estimates and geometric flows.
While classical results focus on specific geometric quantities, the study of weighted inequalities offers a more flexible approach by incorporating weights that can address a broader range of problems. In this paper, we extend these classical inequalities by introducing weights that are convex functions of the potential associated with the metric. Each convex weight function gives rise to a corresponding inequality, reflecting geometric structures such as non-uniform densities or radial symmetries. This work introduces a broad family of Alexandrov-Fenchel and Minkowski type inequalities, which are anticipated to have applications in other geometric and analytic problems.

To set the stage for our results, we first give some definitions of the classical quermassintegrals.

Let $\Omega\subset \mathbb{R}^n$ be a smooth bounded domain enclosed by an embedded hypersurface $\Sigma=\partial\Omega$. The quermassintegrals of $\Omega$ can be expressed (up to a constant) as follows:
\begin{align*}
&W_{-1}(\Omega)=|\Omega|,\qquad W_0(\Omega)=|\Sigma|,\\
&W_k(\Omega)=\int_\Sigma \sigma_k(\kappa)d\mu,\quad k=1,\cdots,n-1,
\end{align*}
where $\sigma_k(\kappa)$ the $k$th mean curvature of $\Sigma$, defined as the $k$th elementary symmetric polynomial of the principal curvatures $\kappa=(\kappa_1,\cdots,\kappa_{n-1})$ of $\Sigma$. The classical Alexandrov-Fenchel inequality states that for convex bounded domain $\Omega\subset \mathbb{R}^n$,
\begin{equation}\label{s1.AFRn}
W_k(\Omega)\geq c_{n,k}\left(W_{k-1}(\Omega)\right)^{\frac{n-1-k}{n-k}},\quad \forall~k=0,\cdots,n-1
\end{equation}
for some constant $c_{n,k}$ with equality holds if and only if $\Omega$ is a round ball. These include the isoperimetric inequality and Minkowski inequality
\begin{equation}\label{s1.MinkR}
\int_\Sigma Hd\mu\geq (n-1)\omega_{n-1}^{\frac{1}{n-1}}|\Sigma|^{\frac{n-2}{n-1}}
\end{equation}
as special cases, where $\omega_{n-1}=|\mathbb{S}^{n-1}|$ is the area of the unit sphere. We refer the readers to Schneider's book \cite{Sch14} for a detailed introduction.

Guan-Li \cite{GL09} showed that the inequalities \eqref{s1.AFRn} still hold for star-shaped and $k$-convex domains in $\mathbb{R}^n$ by using the inverse curvature flows.  See \cite{AFM22,CW13,Mc05,Q15} for proofs of \eqref{s1.AFRn} and \eqref{s1.MinkR} using other methods.

The study of weighted integral inequalities extends classical results by incorporating weights, offering greater flexibility to address a broader range of geometric and analytic problems. Our results, presented in the following two subsections, allow the weight function to be any convex, non-decreasing positive function. The flexibility lies in the fact that the weight functions can vary, with each taking the potential associated with the metric as its variable, and each function giving rise to a corresponding inequality that reflects non-uniform densities or radial symmetries.
To illustrate this flexibility, we provide more than a dozen sharp weighted integral inequalities that can be derived in the two-dimensional space forms, as given in Example \ref{examples}.

Weighted inequalities have been the subject of extensive research in recent decades. Notable examples include the isoperimetric problem with Gaussian density \cite{borell1975brunn} and the isoperimetric problem with radial densities \cite{betta2008weighted}. Additionally, weighted inequalities such as the Caffarelli-Kohn-Nirenberg interpolation inequalities \cite{caffarelli1984first} and weighted Sobolev inequalities \cite{cabre2013Sobolev} have been investigated, highlighting the importance of these weighted integrals.
In our previous work \cite{KW23}, a weighted isoperimetric-type inequality was instrumental in resolving the conjecture that the Weinstock inequality for the first Steklov eigenvalue holds for simply connected mean-convex domains in $\mathbb{R}^n$. The weighted Alexandrov-Fenchel and Minkowski inequalities we establish here are expected to have further applications in analysis, convex geometry, isoperimetric inequalities, and eigenvalue estimates.

\subsection{Weighted Alexandrov-Fenchel inequalities in $\mathbb{R}^n$}
In the first part of the paper, we study the Alexandrov-Fenchel type inequalities in $\mathbb{R}^n$ comparing the weighted curvature integrals and the quermassintegrals. This kind of inequality was firstly studied by the first author and Miao \cite{KM14}, where the inequality
\begin{equation}\label{s1.KM1}
\int_\Sigma r^2Hd\mu\geq n(n-1)|\Omega|
\end{equation}
was proved for star-shaped and mean convex hypersurface $\Sigma \subset \mathbb{R}^n$ which encloses a bounded domain $\Omega$. Equality holds if and only if $\Sigma$ is a round sphere centered at the origin. The inequality \eqref{s1.KM1} has been generalized to higher order mean curvature cases in \cite{GR20,KM15,WZ23}. Recently, Wu \cite{WU24} obtained a further generalization with the weight $r^2$ replaced by $r^{2\alpha}$ for any $\alpha\geq 1$. The results are summarized as following
\begin{thmA}[\cite{GR20,KM15,WZ23,WU24}]
Let $\Sigma$ be a smooth, closed, star-shaped, and $k$-convex hypersurface in $\mathbb{R}^n (n\geq 2)$.  Denote by $\Omega$ the domain enclosed by $\Sigma$. Then, for all $1 \le k \le n-1, -1 \le \ell \le k-1$, and $\alpha \ge 1$, the following inequality holds:
$$
\int_{\Sigma} r^{2 \alpha} \sigma_k d \mu \ge \binom{n-1}{k}\omega_{n-1}\left(\frac{W_\ell(\Omega)}{\binom{n-1}{\ell}\omega_{n-1}}\right)^{\frac{n-1-k+2 \alpha}{n-1-l}},
$$
where $\omega_{n-1}=|\mathbb{S}^{n-1}|$ is the area of the unit round sphere. Moreover, equality holds if and only if $\Sigma$ is a round sphere centered at the origin. (Here $\binom{n-1}{-1}:=\frac{1}{n-1}$.)
\end{thmA}

Our first result is the following Alexandrov-Fenchel type inequalities in $\mathbb{R}^n$ with a general convex weight.

\begin{theorem}\label{s1.thmRn}
Suppose that $\Sigma$ is a smooth closed star-shaped and $k$-convex hypersurface in $\mathbb{R}^n (n\geq 2)$. Denote by $\Omega$ the domain enclosed by $\Sigma$ and $\Phi=r^2/2$. Then, for all $1 \le k \le n-1, -1 \le \ell \le k-1$, and any non-decreasing convex $C^2$ positive function $g$, we have
\begin{equation}\label{s1.eq-thm}
\int_{\Sigma} g(\Phi) \sigma_k d \mu\geq \chi\circ \xi^{-1}\Big(W_\ell(\Omega)\Big),
\end{equation}
where $\chi(r)=\int_{S^{n-1}(r)}g(\Phi)\sigma_k=\binom{n-1}{k}\omega_{n-1}g(\frac{r^2}{2})r^{n-k-1}$ and $\xi(r)=\binom{n-1}{\ell}\omega_{n-1}r^{n-1-\ell}$ are both strictly increasing functions on $\mathbb{R}_+$, and $\xi^{-1}$ denotes the inverse of $\xi(r)$.

If either $g$ is strictly increasing or $g$ is strictly convex, then equality holds in \eqref{s1.eq-thm} if and only if $\Omega$ is a round ball centered at the origin.
\end{theorem}

This result can be restated in a more qualitative form as an isoperimetric-type statement:
\begin{theorem}
Among the class of star-shaped and $k$-convex domains in $\mathbb{R}^n$ with fixed quermassintegral $W_\ell$, where $g$ is a convex and non-decreasing function, the minimum of the weighted curvature integral $\int_{\Sigma} g(\Phi) \sigma_k \, d \mu \, (k > \ell)$ is achieved by the round spheres centered at the origin.
\end{theorem}

For the proof of \eqref{s1.eq-thm}, we apply the normalized inverse curvature flow studied by Gerhardt \cite{Ge90} and Urbas \cite{Ur90} for star-shaped and $k$-convex hypersurfaces in $\mathbb{R}^n$. As a  key step, we show that the weighted curvature integral $\int_{\Sigma} g(\Phi) \sigma_k d \mu$ is monotone non-increasing and the quermassintegral  $W_\ell$ is non-decreasing for $\ell<k$ along the flow.

\subsection{Weighted Minkowski inequality in space forms}

The Minkwoski inequality \eqref{s1.MinkR} is a special case of the Alexandrov inequalities \eqref{s1.AFRn}. Recently, due to the close relationship with the Penrose type inequalities in asymptotically hyperbolic manifolds, the weighted Minkowski type inequality has attracted much attention. In particular, for star-shaped and mean convex hypersurface $\Sigma$ in the hyperbolic space $\mathbb{H}^n$ which encloses a bounded domain $\Omega$, Brendle-Hung-Wang \cite{BHW16}  proved
\begin{equation}\label{s1.BHW}
\int_\Sigma \phi'(r)Hd\mu-n(n-1)\int_\Omega \phi'(r)dv\geq (n-1)\omega_{n-1}^{\frac{1}{n-1}}|\Sigma|^{\frac{n-2}{n-1}},
\end{equation}
where $\phi(r)=\sinh r$. See \cite{DG16} for a related inequality. The Minkowski inequality for a star-shaped and mean convex domain in $\mathbb{H}^n$
\begin{equation}\label{s1.MinkHn}
\int_\Sigma Hd\mu-(n-1)|\Omega|\geq \chi\circ\xi^{-1}\left(|\Sigma|\right)
\end{equation}
was also proved by Brendle-Guan-Li \cite{BGL18} (see also \cite[Theorem 4.2]{GL21}), where $\chi(r)=(n-1)\omega_{n-1}\left(\coth r\sinh^{n-1}r-\int_0^r\sinh^{n-1} sds\right)$ and $\xi(r)=\omega_{n-1}\sinh^{n-1}r$ are both strictly increasing function and $\xi^{-1}$ denotes the inverse of $\xi$. The left hand side of \eqref{s1.MinkHn} in fact represents the first quermassintegral $W_1(\Omega)$ of a bounded domain in $\mathbb{H}^n$. Another version of Minkowski type inequality was established recently by the authors \cite{KW23} for $h$-convex domains in $\mathbb{H}^n$:
\begin{equation}\label{s1.MinkKW}
\int_\Sigma \Phi(r) Hd\mu+(n-1)|\Omega|\geq \chi\circ\xi^{-1}\left(|\Sigma|\right),
\end{equation}
where $\Phi(r)=\int_0^r\phi(s)ds=\cosh r-1$ and $$\chi(r)=(n-1)\omega_{n-1}\left((\cosh r-1)\coth r\sinh^{n-1}r+\int_0^r\sinh^{n-1} sds\right).$$
Note that the left hand side of \eqref{s1.MinkKW} is equivalent to $\int_\Sigma\phi'(r)Hd\mu-W_1(\Omega)$. Here by $h$-convex we mean that all principal curvatures of $\Sigma$ satisfy $\kappa_i\geq 1, i=1,\cdots,n-1$ everywhere on $\Sigma$.

In the second part of this paper, we will discuss various generalizations of the above inequalities in space forms. Before stating our results, we recall some definitions. A space form $M^n(K), n\geq 2,$ is a real simply connected $n$-dimensional Riemannian manifold with constant sectional curvature equal to $K \in\{-1, 0, 1\}$. We view $M^n(K)$ as a warped product manifold $I \times \mathbb{S}^{n-1}$ equipped with the metric
$$
\bar{g}=d r^2+\phi(r)^2 g_{\mathbb{S}^{n-1}}.
$$
There are three cases: $M^n(K)$ is the Euclidean space if $K=0, I=[0,\infty)$ and $\phi(r)=r$; $M^n(K)$ is the hyperbolic space $\mathbb{H}^n$ if $K=-1, I=[0,\infty)$ and $\phi(r)=\sinh r$; and $M^n(K)$ is the sphere $\mathbb{S}^n$ if $K=1, I=[0,\pi)$ and $\phi(r)=\sin r$.
Define $\Phi(r)$ by
$$
\Phi(r)=\int_0^r \phi(s) d s=
\begin{cases}
1-\cos r, & K=1\\
\dfrac{r^2}{2}, & K=0\\
\cosh r-1, & K=-1
\end{cases}
$$
It's well known that $V=\bar{\nabla}\Phi=\phi(r)\partial_r$ is a conformal Killing field on $M^n(K)$ satisfying $\bar{\nabla}V=\phi'(r)\bar{g}$.

Let $g:[0,\infty)\to \mathbb{R}$ be a non-decreasing convex $C^2$ function. For a bounded domain $\Omega$ in a space form $M^n(K)$, we define the weighted volume of $\Omega$
\begin{equation}\label{s3.Ag}
A_g(\Omega)=\int_\Omega \left(g'(\Phi)\phi'+Kg(\Phi) \right)dv.
\end{equation}
When $g(x)=x$, we have $g'(\Phi)\phi'+Kg(\Phi)=\phi'+K\Phi=1$. Then $A_g(\Omega)=\text{vol}(\Omega)$ is just the volume of the domain $\Omega$. When $g(x)\equiv 1$, then $A_g(\Omega)=K\text{vol}(\Omega)$ is also the volume of $\Omega$ up to a constant $K$.

We first prove an inequality for convex curves in the two dimensional case.
\begin{theorem}\label{s1.thmn=2}
For a smooth, closed, strictly convex curve $\gamma = \partial \Omega$ on $M^2(K)$, with $\gamma$ enclosing the origin, and a non-decreasing convex $C^2$ function $g$, we have
\begin{align}\label{s1.ineqn=2}
\int_\gamma g(\Phi) \kappa d s+A_g(\Omega) \ge 2 g\left(\frac{\mathrm{A}(L)}{2 \pi}\right) \sqrt{4 \pi^2-K L^2}-2 \pi g(0)+4 \pi K G\left(\frac{\mathrm{A}(L)}{2 \pi}\right).
\end{align}
Here $\kappa, L$ are the curvature and length of the curve $\gamma$ respectively, $\mathrm{A}(l):=$ the area of the (convex) geodesic ball in $M^2(K)$ with circumference $l$ and $G(r):=\int_{0}^{r}g(t)dt$.

If either $g$ is strictly increasing or $g$ is strictly convex, then the equality holds if and only if $\gamma$ is a geodesic circle centered at the origin.
\end{theorem}
\begin{example}\label{examples}
For illustration, we present examples of sharp inequalities involving three geometric quantities that can be derived for strictly convex curves in $\mathbb{R}^2$ by applying various choices of the function $g$ in \eqref{s1.ineqn=2}:
\begin{enumerate}
\item $\int _\gamma \frac{r^2}{2} \kappa {ds}+|\Omega|\ge \frac{L^2}{2 \pi}$.
\item $\int _\gamma \frac{r^3}{3} \kappa {ds}+\int _\Omega { r} dv \ge \frac{L^3}{6 \pi^2}$.
\item $\int _\gamma \frac{r^4}{4} \kappa {ds}+\int _\Omega r^2dv \ge \frac{L^4}{16 \pi^3}$.
\item $\int_ \gamma e^{\frac{r^2}{2}} \kappa d s+\int_\Omega e^{\frac{r^2}{2}} d v\ge 2 \pi\left(2 e^{\frac{L^2}{8 \pi^2}}-1\right)$.
\item
$\int_ \gamma e^{\frac{r^4}{4}} \kappa d s+\int_\Omega e^{\frac{r^2}{2}} d v\ge2 \pi\left(2 e^{\frac{L^4}{64 \pi^4}}-1\right)$.
\item
$\int _\gamma \sinh \left(\frac{r^2}{2}\right) \kappa  {ds}+\int_\Omega \cosh \left(\frac{r^2}{2}\right) d v\ge 4 \pi \sinh \left(\frac{L^2}{8 \pi^2}\right)$.
\item
$\int _\gamma \cosh \left(\frac{r^2}{2}\right) \kappa  {ds}+\int_\Omega \sinh \left(\frac{r^2}{2}\right) d v\ge
4\pi \cosh \left(\frac{L^2}{8 \pi^2}\right)-2\pi$.
\end{enumerate}
If $\gamma$ is a strictly convex curve enclosing the center of an open geodesic ball of radius $\frac{\pi}{2}$ in $\mathbb{S}^2$, then the following sharp inequalities hold:
\begin{enumerate}
\item
$\int_\gamma (1-\cos r)\kappa ds +|\Omega|\ge \frac{L^2}{2 \pi}$.
\item
$\int _\gamma (1-\cos r)^2 \kappa d s+\int_\Omega \sin ^2 r d v\ge \frac{4}{3}\left(\sqrt{4 \pi^2-L^2}-2 \pi\right)-\frac{L^2}{3 \pi^2}\left(\sqrt{4 \pi^2-L^2}-3 \pi\right)$.
\item
$\int _\gamma (1-\cos r)^3 \kappa d s+\int_\Omega
(1-\cos r)^2(2 \cos r +1) dv\ge
-\frac{3 L^4}{16 \pi^3}-\frac{L^2}{\pi^2}\left(\sqrt{4 \pi^2-L^2}-3 \pi\right)+4 \sqrt{4 \pi^2-L^2}-8 \pi$
\item
$\int_\gamma\left(\sec r -1\right)\kappa ds+\int_\Omega (2\sec r -1) dv\ge 4 \pi \log \left(\frac{2 \pi}{\sqrt{4 \pi^2-L^2}}\right)$.
\item
$\int_\gamma\left(\cos r +\sec r -2\right)\kappa ds+\int_\Omega (2\sec r -2) dv\ge4 \pi \log \left(\frac{2 \pi}{\sqrt{4 \pi^2-L^2}}\right)-\frac{L^2}{2 \pi}$.
\end{enumerate}
These inequalities are derived by substituting $g(x)$ with $x, x^2, x^3$,  $\frac{x}{1-x}$ and $\frac{x^2}{1-x}$, respectively. Note that the last two functions are convex over the range of $\Phi$.

Similarly, in $\mathbb H^2$, we have the following sharp inequalities for strictly convex curves:
\begin{enumerate}
\item
$\int_\gamma (\cosh r -1)\kappa ds +|\Omega|\ge \frac{L^2}{2 \pi}$.
\item
$\int _\gamma (\cosh r -1)^2 \kappa d s+\int_\Omega \sinh ^2 r d v\ge \frac{L^2}{3\pi^2}\left(\sqrt{L^2+4 \pi^2}-3 \pi\right) +\frac{4 }{3}\left(\sqrt{L^2+4 \pi^2}-2 \pi\right) $.
\item
$\int _\gamma (\cosh r -1)^3 \kappa d s+\int_\Omega
(\cosh r -1)^2(2 \cosh r +1) dv\ge \frac{3 L^4}{16 \pi^3}-\frac{L^2}{\pi^2}\left(\sqrt{L^2+4 \pi^2}-3 \pi\right)-4 \sqrt{L^2+4 \pi^2}+8 \pi $.
\item
$\int_\gamma\left(\cosh r + \textrm{sech } r -2\right)\kappa ds+\int_\Omega (2-2 \textrm{sech } r ) dv\ge \frac{L^2}{2 \pi}-2 \pi \log \left(\frac{L^2}{4 \pi^2}+1\right)$.
\end{enumerate}
These inequalities are obtained by substituting $g(x)$ as $x, x^2, x^3$, and $\frac{x^2}{1+x}$, respectively.
\end{example}

When $g = 1$, the inequality in Theorem \ref{s1.thmn=2} reduces to an equality, corresponding to the classical Gauss-Bonnet formula. The first inequality in the three lists above was proved by the authors and G. Wheeler and V.-M.Wheeler in \cite[Theorem 1.2]{KWWW22} for the specific case where $g(x) = x$, and was applied in \cite[Corollary 5.8]{KWWW22}  to provide a counterexample to the $n = 2$ case of a conjecture by Girão and Pinheiro \cite{GP17}.

To prove Theorem \ref{s1.thmn=2}, we employ the inverse curvature type flow $X: \mathbb{S}^1 \times[0, T) \rightarrow M^2(K)$
\begin{equation}\label{flow}
\begin{cases}
\frac{\partial}{\partial t} X(x, t)=\left(\dfrac{\phi^{\prime}(r)}{\kappa}-u\right) \nu(x, t) \\
X(\cdot, 0)=X_0,
\end{cases}
\end{equation}
where $\kappa$ is the curvature, $\nu$ is the unit outward normal, and $u=\langle V, \nu\rangle$ is the support function of the curve. The flow \eqref{flow} for strictly convex closed curves in space forms has been studied in \cite{KWWW22}, and it turns out to have a number of applications, including a unified proof of the classical isoperimetric inequality in $2$-dimensional space forms. For initially smooth strictly convex closed curve $\gamma_0$, the flow \eqref{flow} preserves the convexity, exists for all time and converges to a geodesic circle as $t\to\infty$. A nice feature of this flow is that it preserves the length of the evolving curves. A key in proving Theorem \ref{s1.thmn=2} is to show that the left hand side of \eqref{s1.ineqn=2} is monotone decreasing along the flow \eqref{flow}.

For the higher dimensional case, we prove the following Minkowski-type inequalities for $h$-convex hypersurfaces in the hyperbolic space $\mathbb{H}^n$ and for convex hypersurfaces in the sphere $\mathbb{S}^n$.
\begin{theorem}\label{s1.thmH}
Suppose that $\Sigma$ is a smooth closed $h$-convex hypersurface in $\mathbb{H}^n (n\geq 3)$ enclosing a bounded domain $\Omega$ that contains the origin. For any non-decreasing convex $C^2$ positive function $g$, we have
\begin{equation}\label{s1.eq-thmH}
\int_{\Sigma} g(\Phi) H d \mu+(n-1)A_g(\Omega)\geq \chi\circ \xi^{-1}\big(|\Sigma|\big),
\end{equation}
where $\xi(r)=\omega_{n-1}\phi^{n-1}(r)$, and $\chi(r)=\int_{S(r)}g(\Phi)H+(n-1)A_g(B(r))$.  If either $g$ is strictly increasing or $g$ is strictly convex, then equality holds in \eqref{s1.eq-thmH} if and only if $\Omega$ is a geodesic ball centered at the origin.
\end{theorem}

\begin{remark}
\begin{enumerate}
\item When $g(x)\equiv 1$, the left hand side of \eqref{s1.eq-thmH} reduces to the first quermassintegral $W_1(\Omega)=\int_\Sigma Hd\mu-(n-1)\text{vol}(\Omega)$. In this case, Theorem \ref{s1.thmH} reduces to the Minkowski inequality \eqref{s1.MinkHn} in hyperbolic space proved by Bendle-Guan-Li \cite{BGL18}. See Hu-Li-Wei \cite{HLW22} for an alternative proof for $h$-convex domains in $\mathbb{H}^n$.
\item When $g(x)=x$, then $A_g(\Omega)=\text{vol}(\Omega)$ and the left hand side of \eqref{s1.eq-thmH} reduces to $\int_\Sigma\Phi Hd\mu+(n-1)\text{vol}(\Omega)$. In this case, Theorem \ref{s1.thmH} reduces to the inequality \eqref{s1.MinkKW} proved previously by the authors \cite{KW23}.
\end{enumerate}
\end{remark}

\begin{theorem}\label{s1.thmSn}
Suppose that  $\Sigma$ is a smooth strictly convex hypersurface in $\mathbb{S}^n$ enclosing a bounded domain $\Omega$ that contains the origin. For any non-decreasing convex $C^2$ positive function $g$, we have
\begin{equation}\label{s1.eq-thmSn}
\int_{\Sigma} g(\Phi) H d \mu+(n-1)A_g(\Omega)\geq \chi\circ \zeta^{-1}\Big(\int_\Omega\phi'(r)dv\Big),
\end{equation}
where $\zeta(r)=\int_{B(r)} \phi'dv=\frac{1}{n}\omega_{n-1}\sin^nr$, and $\chi(r)=\int_{S(r)}g(\Phi)H+(n-1)A_g(B(r))$.  If either $g$ is strictly increasing or $g$ is strictly convex, then equality holds in \eqref{s1.eq-thmSn} if and only if $\Omega$ is a geodesic ball centered at the origin.
\end{theorem}

The inequalities \eqref{s1.eq-thmH} and \eqref{s1.eq-thmSn} will be proved using inverse mean curvature type flow. Due to the opposite sign of the ambient curvatures, we employ different flows in the hyperbolic space and in sphere separately. This is crucial in the proof of monotonicity of $\int_{\Sigma} g(\Phi) H d \mu+(n-1)A_g(\Omega)$ along the flows. To prove \eqref{s1.eq-thmH}, we use the flow
\begin{equation}\label{s4.Flowk=1}
\frac{\partial}{\partial t}X=\left(\frac{\phi'}{H}-\frac{u}{n-1}\right)\nu,
\end{equation}
which was introduced by Brendle, Guan and Li \cite{BGL18} (see also \cite[\S 4]{GL21}). The second author and Hu, Li \cite{HLW22} proved that for $h$-convex hypersurface in $\mathbb{H}^n$, the flow \eqref{s4.Flowk=1} preserves the $h$-convexity and converges to a geodesic sphere as $t\to\infty$. While for proving \eqref{s1.eq-thmSn}, we use the flow
\begin{equation}\label{s5.Flowk=1}
\frac{\partial}{\partial t}X=\left(\frac{n-1}{H}-\frac{u}{\phi'}\right)\nu
\end{equation}
which was studied by Scheuer and Xia \cite{SX19} and evolves any strictly convex hypersurface in sphere $\mathbb{S}^n$ to a geodesic sphere as $t\to\infty$.

As an application of the weighted inequalities derived in our work, we establish a sharp upper bound for the first non-zero eigenvalue of a class of differential operators associated with $k$-convex hypersurfaces $\Sigma$ in $\mathbb{R}^n$. For simplicity, we only state the result for the two-dimensional case. Readers are referred to Theorem \ref{eigen} for the general result in higher dimensions.

\begin{theorem}
For a smooth, closed, strictly convex curve $\gamma$ in $\mathbb{R}^2$, the first non-zero eigenvalue $\lambda_1$ of the operator $-\frac{1}{\kappa} \Delta$ on $\gamma$ satisfies
$$\lambda_1 \le \frac{\pi L}{L^2 - 2\pi A}, $$
where $\kappa, L$ denote the curvature and length of $\gamma$ respectively, and $A$ represents the area enclosed by $\gamma$. Equality holds if and only if $\gamma$ is a circle.
\end{theorem}
\begin{ack}
The first author was supported by the UOW Early-Mid Career Researcher Enabling Grant and the UOW Advancement and Equity Grant Scheme for Research 2024. The second author was supported by National Key Research and Development Program of China 2021YFA1001800 and 2020YFA0713100, and the Fundamental Research Funds for the Central Universities.
\end{ack}

\section{Preliminaries}

In this section, we review some definitions and basic formulas for hypersurfaces in space forms, and then derive a general evolution equation for the weighted curvature integral $\int_{\Sigma} g(\Phi) \sigma_k d \mu$.

\subsection{Hypersurfaces in space forms}
Given a smooth bounded domain $\Omega$ in a space form $M^n(K)$, its boundary $\Sigma=\partial\Omega$ is a smooth embedded hypersurface. The support function of $\Sigma$ is defined by $u=\langle V,\nu\rangle$, where $\nu$ is the outward unit normal of $\Sigma$ and $V=\bar{\nabla}\Phi$ is the conformal Killing field on $M^n(K)$. Note that in Euclidean space, $V$ is just the position vector.  We denote by $\kappa=(\kappa_1,\cdots,\kappa_{n-1})$ the principal curvatures of $\Sigma$, which are defined as the eigenvalues of the Weingarten tensor $h_i^j=g^{jk}h_{ik}$. If we denote $\sigma_k$ as the $k$th elementary symmetric polynomial of an $n$-vector, then $H=\sigma_1(\kappa)=\kappa_1+\cdots+\kappa_{n-1}$ is the mean curvature of $\Sigma$, and $\sigma_k(\kappa)$ is called the $k$th mean curvature of $\Sigma$. The definition also extends to the Weingarten tensor by defining $\sigma_k(h_i^j)=\sigma_k(\lambda(h_i^j))$, where $\lambda$ is a map sending a symmetric matrix to its eigenvalues. The Newton tensor of $h_i^j$ is defined by $T_k^{ij}=\frac{\partial\sigma_{k+1}}{\partial h_{i}^j}$ and satisfies the iteration (see \cite{Re73})
\begin{equation}\label{s2.NT0}
T_k^{ij}=\sigma_kg^{ij}-T_{k-1}^{ik}h_{k}^j.
\end{equation}
It is divergence free due to the Codazzi equation in space forms.
\begin{lemma}[see \S 2 in \cite{HS99}]\label{s2.lemsk}
For $k=1,\cdots,n-1$, we have
\begin{align}
\sum_{ij}T_{k-1}^{ij}h_{ij}=&k\sigma_k\\
\sum_{ij}T_{k-1}^{ij}h_{ij}=&(n-k)\sigma_{k-1}\\
\sum_{ij}T_{k-1}^{ij}(h^2)_{ij}=&\sigma_1\sigma_k-(k+1)\sigma_{k+1},
\end{align}
where $(h^2)_i^j=\sum_{k=1}^{n-1}h_i^kh_k^j$ and by convention we set $\sigma_k=0$ for $k>n-1$.
\end{lemma}
The support function $u$ and $\Phi$ satisfy the following formulae (see \cite[\S 2]{GL15})
\begin{lemma}
Let $\Sigma$ be a smooth closed hypersurface in the space form $M^n(K)$. Then
\begin{align}
&\nabla_iu=h_i^j\nabla_j\Phi\label{s2.du}\\
& \Delta \Phi=(n-1)\phi'(r)-uH,\label{s2.DeltaPhi}\\
&\nabla_i\left(T_k^{ij}\nabla_j\Phi\right)=(n-1-k)\phi'\sigma_k(\kappa)-(k+1)u\sigma_{k+1}(\kappa). \label{s2.dT1}
\end{align}
where $\nabla, \Delta$ denote the Levi-Civita connection and Laplacian operator on $\Sigma$ with respect to the induced metric.
\end{lemma}

A smooth closed hypersurface $\Sigma$ is called $k$-convex if its principal curvatures are contained in the Garding cone $\Gamma_k^+$ defined by
\begin{equation*}
\Gamma_+^+=\{\kappa\in\mathbb{R}^{n-1},~|~\sigma_j(\kappa)>0,\quad \forall~j=1,\cdots,k\}.
\end{equation*}
Denote
\begin{equation}\label{s2.Hk}
H_k(\kappa)=\frac{\sigma_k(\kappa)}{\binom{n-1}{k}},\quad k=1,\cdots, n-1
\end{equation}
the $k$th normalized mean curvature of $\Sigma$. We have the following Newton-MacLaurin inequality (see \S 2 in \cite{Gu14}).
\begin{lemma}
For any integers $1\leq \ell\leq k\leq n-1$ and $\kappa\in \Gamma_k^+$, we have
\begin{align}
&H_{k+1}H_{\ell-1}\leq H_k H_\ell, \label{s2.NT1}\\
& H_{k+1}\leq H_k^{\frac{k+1}{k}}.\label{s2.NT2}
\end{align}
Equality holds if and only if $\kappa=c(1,\cdots,1)$ for some constant $c>0$.
\end{lemma}

\subsection{Evolution equation}
Let $X_0: \mathbb{S}^{n-1} \rightarrow M^n(K)$ be a smooth, closed star-shaped hypersurface. If $K=1$, we also assume $X_0$ lies in the open hemisphere $\mathbb{S}^n_+$. We consider the flow $X: \mathbb{S}^{n-1} \times[0, T) \rightarrow M^n(K)$ satisfying
\begin{equation}\label{s3.flow-F}
\frac{\partial}{\partial t}X=F\nu,
\end{equation}
where the speed function $F$ depends on the curvature and the position of $\Sigma_t=X(\mathbb{S}^{n-1},t)$. The following lemma gives a variational equation of the weighted curvature integral along the flow \eqref{s3.flow-F}. See Reilly \cite[Theorem A]{Re73} for a general variational formula in the Euclidean space for curvature functions involving potential $\Phi$ and support function.
\begin{lemma}\label{s2.lem-nge3}
For $n\geq 2$, along the flow \eqref{s3.flow-F}, we have
\begin{align}\label{s2.lemevlsk1}
& \frac{d}{d t} \int_{\Sigma_t} g(\Phi) \sigma_k d\mu_t\nonumber\\
& =\int_{\Sigma}\Big(g'(\Phi) u \sigma_k-g'(\Phi)\nabla_i\left(T_{k-1}^{ij}\nabla_j\Phi\right)-g''(\Phi)T_{k-1}^{ij}\nabla_i\Phi\nabla_j\Phi\nonumber\\
&\qquad+ (k+1)g(\Phi) \sigma_{k+1}-(n-k)Kg(\Phi)\sigma_{k-1}\Big) Fd\mu_t
\end{align}
for $k=1,\cdots,n-1$. Furthermore, for the mean curvature case $k=1$, we have
\begin{align}\label{s3.evl1}
& \frac{d}{d t} \left(\int_{\Sigma_t} g(\Phi) Hd\mu_t+(n-1)A_g(\Omega_t)\right)\nonumber\\
=&\int_{\Sigma_t}\biggl(2g'(\Phi) u H+2g(\Phi)\sigma_2(\kappa)-g''(\Phi)|\nabla\Phi|^2\biggr)Fd\mu_t.
\end{align}
By convention, $\sigma_k=0$ for $k>n-1$.
\end{lemma}
\proof
By the definition of $\Phi$, we first have that along \eqref{s3.flow-F}
\begin{equation*}
\partial_t\Phi=\langle \phi(r)\partial_r, \partial_tX\rangle =Fu.
\end{equation*}
The evolution of Weingarten tensor $h_i^j$ and area form $d\mu_t$ is well known (see \cite[Theorem 3-15]{And94}):
\begin{align*}
\partial_th_i^j=&-\nabla^j\nabla_i F-F\left((h^2)_i^j+K\delta_i^j \right)\\
\partial_td\mu_t=&F\sigma_1d\mu_t.
\end{align*}
Then using Lemma \ref{s2.lemsk} and the divergence-free property of $T_{k-1}^{ij}$, we have
\begin{align}\label{s2.evlsk1}
& \frac{d}{d t} \int_{\Sigma_t} g(\Phi) \sigma_k d\mu_t\nonumber\\
& =\int_{\Sigma_t}\left(\frac{\partial}{\partial t} g(\Phi)\right) \sigma_k d\mu_t+\int_{\Sigma_t} g(\Phi) \frac{\partial \sigma_k}{\partial t} d\mu_t+\int_{\Sigma_t} g(\Phi) \sigma_k\sigma_1Fd\mu_t\nonumber\\
& =\int_{\Sigma_t} g'(\Phi) u F \sigma_kd\mu_t+\int_{\Sigma_t} g(\Phi)T_{k-1}^{ij}\left(-\nabla_i\nabla_j F-F\left((h^2)_{ij}+Kg_{ij}\right)\right)d\mu_t\nonumber\\
&\qquad +\int_{\Sigma_t} g(\Phi) F \sigma_1\sigma_k d\mu_t\nonumber\\
& =\int_{\Sigma}\Big(g'(\Phi) u \sigma_k-\nabla_i\left(T_{k-1}^{ij}\nabla_j g(\Phi)\right)+ g(\Phi)\left(\sigma_1\sigma_k-T_{k-1}^{ij}((h^2)_{ij}+Kg_{ij})\right)\Big) F d\mu_t\nonumber\\
& =\int_{\Sigma}\Big(g'(\Phi) u \sigma_k-g'(\Phi)\nabla_i\left(T_{k-1}^{ij}\nabla_j\Phi\right)-g''(\Phi)T_{k-1}^{ij}\nabla_i\Phi\nabla_j\Phi\nonumber\\
&\qquad+ (k+1)g(\Phi) \sigma_{k+1}-(n-k)Kg(\Phi)\sigma_{k-1}\Big) Fd\mu_t.
\end{align}
This proves the equation \eqref{s2.lemevlsk1}.

For $k=1$, using \eqref{s2.dT1}, we can further rewrite the above equation and obtain
\begin{align}\label{s3.evl-gH}
\frac{d}{d t} \int_{\Sigma_t} g(\Phi) Hd\mu_t=&\int_{\Sigma}\biggl(2g'(\Phi) u H+2g(\Phi)\sigma_2-g''(\Phi)|\nabla\Phi|^2\nonumber\\
&\quad -(n-1)\Big(\phi'g'(\Phi)+g(\Phi) K\Big) \biggr) Fd\mu_t.
\end{align}
By the definition \eqref{s3.Ag} of $A_g$, the co-area formula implies that
\begin{equation}\label{s3.evl-Ag}
\frac{d}{dt}A_g(\Omega_t)=\int_{\Sigma_t}\left(g'(\Phi)\phi'+Kg(\Phi) \right)Fd\mu_t.
\end{equation}
Combining \eqref{s3.evl-gH} and \eqref{s3.evl-Ag} gives \eqref{s3.evl1}.
\endproof

\section{Weighted Alexandrov-Fenchel type inequalities in $\mathbb{R}^n$}\label{sec.3}
In this section, we first consider the Euclidean space case. We prove Theorem \ref{s1.thmRn} on the weighted Alexandrov-Fenchel type inequality with convex weight for hypersurfaces in $\mathbb{R}^n$.

For convenience of notations, in this section we use the normalized $k$th mean curvature $H_k=\binom{n-1}{k}^{-1}\sigma_k$. We will use the normalized inverse mean curvature flow
\begin{equation}\label{s4.IMCF}
\frac{\partial}{\partial t}X=\left(\frac{H_{k-1}}{H_k}-u\right)\nu
\end{equation}
to prove \eqref{s1.eq-thm}. This flow is equivalent to the inverse curvature flow
\begin{equation*}
\frac{\partial}{\partial t}X=\frac{H_{k-1}}{H_k}\nu
\end{equation*}
introduced by Gerhardt \cite{Ge90} and Urbas \cite{Ur90}, up to a rescaling. If the initial hypersurface is star-shaped and $k$-convex, the flow \eqref{s4.IMCF} converges to a round sphere as $t\to\infty$.  Note that when $n=2$, a closed convex curve in $\mathbb{R}^2$ is star-shaped with respect to some point centered inside the enclosed domain and is $k$-convex for $k=n-1=1$.

In the Euclidean space, $\phi'(r)=1$, then \eqref{s2.dT1} is equivalent to
\begin{equation}\label{s4.dTk}
\nabla_i\left(T_{k-1}^{ij}\nabla_j\Phi\right)=k\binom{n-1}{k}\left(H_{k-1}-uH_k\right).
\end{equation}

\begin{proof}[Proof of Theorem \ref{s1.thmRn}]
By the general evolution equation \eqref{s2.lemevlsk1} and noting that the ambient curvature $K=0$ in $\mathbb{R}^n$, along the flow \eqref{s4.IMCF}, we have
\begin{align}
& \frac{d}{d t} \int_{\Sigma_t} g(\Phi) H_k d\mu_t\nonumber\\
& =\int_{\Sigma_t}\Bigg(g'(\Phi) u H_k-\frac{1}{\binom{n-1}{k}}g'(\Phi)\nabla_i\left(T_{k-1}^{ij}\nabla_j\Phi\right)\nonumber\\
&\qquad-\frac{1}{\binom{n-1}{k}}g''(\Phi)T_{k-1}^{ij}\nabla_i\Phi\nabla_j\Phi+ (n-1-k)g(\Phi) H_{k+1}\Bigg) \times \left(\frac{H_{k-1}}{H_k}-u\right)d\mu_t\nonumber\\
&=\int_{\Sigma_t}\Bigg(kg'(\Phi) u\left(H_{k-1}-uH_k\right)d\mu_t -(k-1)\int_{\Sigma_t}g'(\Phi)H_{k-1}\left(\frac{H_{k-1}}{H_k}-u\right)d\mu_t\nonumber\\
&\quad -\int_{\Sigma_t}\frac{g'(\Phi)}{H_k}\left(H_{k-1}-uH_k\right)^2d\mu_t\nonumber\\
&\quad -\frac{1}{\binom{n-1}{k}}\int_{\Sigma_t}g''(\Phi)T_{k-1}^{ij}\nabla_i\Phi\nabla_j\Phi\left(\frac{H_{k-1}}{H_k}-u\right)d\mu_t\nonumber\\
&\quad + (n-1-k)\int_{\Sigma_t}g(\Phi)H_{k+1}\left(\frac{H_{k-1}}{H_k}-u\right)d\mu_t\nonumber\\
=&I+II+III+IV+V,
\end{align}
where in the second equality we used \eqref{s4.dTk}. Notice that we have $\sigma_k=0$ for $k>n-1$. This includes the case $n=2$.  We estimate the five terms $I, \cdots, V$ separately.  Since $g'\geq 0$ and $H_k>0$ on $\Sigma_t$, the third term $III$ is clearly non-positive. Using \eqref{s4.dTk} and \eqref{s2.du}, we can estimate the first term $I$ as follows:
\begin{align*}
I=&\frac{1}{\binom{n-1}{k}}\int_{\Sigma_t}g'u\nabla_i\left(T_{k-1}^{ij}\nabla_j\Phi\right)\\
=&-\frac{1}{\binom{n-1}{k}}\int_{\Sigma_t}T_{k-1}^{ij} \nabla_i\Phi \left(g''u\nabla_j\Phi+g'h_j^k\nabla_k\Phi \right)\\
=&-\frac{1}{\binom{n-1}{k}}\int_{\Sigma_t}ug''T_{k-1}^{ij}\nabla_i\Phi\nabla_j\Phi-\frac{1}{\binom{n-1}{k}}\int_{\Sigma_t}g'T_{k-1}^{ij}h_j^k\nabla_i\Phi\nabla_i\Phi.
\end{align*}
Since the Newton operator satisfies $$T_k^{ij}=\sigma_kg^{ij}-T_{k-1}^{ik}h_{k}^j,$$ we can rewrite the second term and obtain
\begin{align*}
I=&-\frac{1}{\binom{n-1}{k}}\int_{\Sigma_t}ug''T_{k-1}^{ij}\nabla_i\Phi\nabla_j\Phi-\frac{1}{\binom{n-1}{k}}\int_{\Sigma_t}g'(\Phi)\sigma_k|\nabla\Phi|^2d\mu_t\\
&\qquad +\frac{1}{\binom{n-1}{k}}\int_{\Sigma_t}g'(\Phi)T_k^{ij}\nabla_i\Phi\nabla_j\Phi d\mu_t.
\end{align*}
As $T_k$ is divergence free, by integration by parts and using \eqref{s4.dTk}, we have
\begin{align*}
& \frac{1}{\binom{n-1}{k}}\int_{\Sigma_t}g'(\Phi)T_k^{ij}\nabla_i\Phi\nabla_j\Phi \\
=& \frac{1}{\binom{n-1}{k}}\int_{\Sigma_t}T_{k}^{ij}\nabla_ig(\Phi)\nabla_j\Phi\\
=& - \frac{1}{\binom{n-1}{k}}\int_{\Sigma_t}g(\Phi)\nabla_i\left(T_k^{ij}\nabla_j\Phi\right)\\
=& -(n-1-k)\int_{\Sigma_t}g(\Phi)\left(H_k-uH_{k+1}\right).
\end{align*}
Therefore,
\begin{align}\label{s4.I}
I=&-\frac{1}{\binom{n-1}{k}}\int_{\Sigma_t}ug''T_{k-1}^{ij}\nabla_i\Phi\nabla_j\Phi-\frac{1}{\binom{n-1}{k}}\int_{\Sigma_t}g'(\Phi)\sigma_k|\nabla\Phi|^2d\mu_t\nonumber\\
& -(n-1-k)\int_{\Sigma_t}g(\Phi)\left(H_k-uH_{k+1}\right).
\end{align}

When $k=1$, the second term $II$ vanishes. When $k>1$,  by the Newton-MacLaurin inequality  \eqref{s2.NT1} and noting that $g'\geq 0$, the second term satisfies
\begin{align}
II\leq &-(k-1)\int_{\Sigma_t}g'(\Phi)\left(H_{k-2}-uH_{k-1}\right)\\
=& -\frac{1}{\binom{n-1}{k-1}}\int_{\Sigma_t}g'(\Phi)\nabla_i\left(T_{k-2}^{ij}\nabla_j\Phi\right)\nonumber\\
=& \frac{1}{\binom{n-1}{k-1}}\int_{\Sigma_t}g''(\Phi)T_{k-2}^{ij}\nabla_i\Phi\nabla_j\Phi.
\end{align}
Since $g\geq 0$, the Newton-MacLaurin inequality \eqref{s2.NT1} also implies that
\begin{align}\label{s4.V}
V\leq & (n-1-k)\int_{\Sigma_t}g(\Phi)\left(H_{k}-uH_{k+1}\right)d\mu_t.
\end{align}

Combining the above equations, we have
\begin{align*}
\frac{d}{d t} \int_{\Sigma_t} g(\Phi) H_k d\mu_t
\leq &-\frac{1}{\binom{n-1}{k}}\int_{\Sigma_t}g'(\Phi)\sigma_k|\nabla\Phi|^2d\mu_t\\
&\quad +\frac{1}{\binom{n-1}{k-1}}\int_{\Sigma_t}g''(\Phi)T_{k-2}^{ij}\nabla_i\Phi\nabla_j\Phi\\
&\quad -\frac{1}{\binom{n-1}{k}}\int_{\Sigma_t}g''(\Phi)T_{k-1}^{ij}\nabla_i\Phi\nabla_j\Phi\frac{H_{k-1}}{H_k}d\mu_t\\
\leq &-\int_{\Sigma_t}g''(\Phi)\biggl(\frac{1}{\binom{n-1}{k}}T_{k-1}^{ij}\frac{H_{k-1}}{H_k}-\frac{1}{\binom{n-1}{k-1}}T_{k-2}^{ij}\biggr)\nabla_i\Phi\nabla_j\Phi d\mu_t\\
=&\int_{\Sigma_t}g''(\Phi)\frac{H_{k-1}^2}{H_k}g^{js}\frac{\partial}{\partial h_{i}^s}\left(\frac{H_k}{H_{k-1}}\right)\nabla_i\Phi\nabla_j\Phi d\mu_t\\
\leq &0,
\end{align*}
which follows from the fact that $H_k>0$, $g'\geq 0$, $g''\geq 0$ and that $H_k/H_{k-1}$ is strictly increasing with respect to the Weingarten tensor $h_{i}^j$.  This means that $\int_{\Sigma_t} g(\Phi) \sigma_k d\mu_t$
is strictly decreasing along the flow \eqref{s4.IMCF} unless $\Sigma_t$ is totally umbilic.

On the other hand, by the evolution equation of the quermassintegral $W_\ell(\Omega_t)=\int_{\Sigma_t}\sigma_\ell d\mu_t$ (a special case of the equation \eqref{s2.lemevlsk1}), we have for $\ell=0,1,\cdots,k-1$,
\begin{align*}
\frac{d}{d t} W_\ell(\Omega_t)=&(\ell+1)\int_{\Sigma_t} \sigma_{\ell+1} \left(\frac{H_{k-1}}{H_k}-u\right)d\mu_t\\
=&(\ell+1)\binom{n-1}{\ell+1}\int_{\Sigma_t}\left(\frac{H_{\ell+1}H_{k-1}}{H_k}-uH_{\ell+1}\right)d\mu_t\\
\geq &(\ell+1)\binom{n-1}{\ell+1}\int_{\Sigma_t}\left(H_\ell-uH_{\ell+1}\right)d\mu_t\\
=&\int_{\Sigma_t}\nabla_i\left(T_{\ell}^{ij}\nabla_j\Phi\right)d\mu_t\\
=&0
\end{align*}
along the flow \eqref{s4.IMCF}, where we used \eqref{s4.dTk} and the Newton-MacLaurin inequality \eqref{s2.NT1}. For $\ell=-1$, by the coarea formula and the Heintze-Karcher inequality \cite{Ros87}, we also have
\begin{align*}
\frac{d}{d t} W_{-1}(\Omega_t)=&\int_{\Sigma_t} \left(\frac{H_{k-1}}{H_k}-u\right)d\mu_t\\
\geq &\int_{\Sigma_t}\left(\frac{n-1}{H}-u\right)d\mu_t\\
\geq & 0.
\end{align*}
That is, $W_\ell(\Omega_t)$, $\ell=-1,0,\cdots,k-1$ are all non-decreasing along the flow \eqref{s4.IMCF}.

Since the flow converges to a round sphere $S^{n-1}(r)$ as $t\to\infty$, we have
\begin{align*}
W_\ell(\Omega)\leq &\lim_{t\to\infty}W_\ell(\Omega_t)=\binom{n-1}{\ell}\omega_{n-1}r^{n-1-\ell}=:\xi(r)\\
\int_{\Sigma} g(\Phi) \sigma_k d\mu\geq &\lim_{t\to\infty}\int_{\Sigma_t} g(\Phi) \sigma_k d\mu_t=\binom{n-1}{k}\omega_{n-1}g(\frac{r^2}{2})r^{n-k-1}=:\chi(r).
\end{align*}
Both $\xi(r)$ and $\chi(r)$ are strictly increasing in $r$ and so the inverse $\xi^{-1}$ of $\xi(r)$ is well defined. Combining the above two inequalities, we have
\begin{equation*}
\int_{\Sigma} g(\Phi) \sigma_k d\mu\geq \chi\circ \xi^{-1}\Big(W_\ell(\Omega)\Big).
\end{equation*}
This proves the inequality \eqref{s1.eq-thm}.

If the equality holds, from the proof of the monotonicity, $\Sigma_t$ is totally umbilic for all $t$, implying that $\Sigma$ is a round sphere. Furthermore, when either $g$ is strictly increasing or $g$ is strictly convex, we see that $|\nabla\Phi|$ vanishes identically and so $\Sigma$ is a round sphere centered at the origin.
\end{proof}

\section{Weighted Minkowski type inequalities in space forms}
In this section, we will focus on the Minkowski type inequalities with general convex weight in space forms.

\subsection{$2$-dimensional case}\label{sec.curve}
We first treat the case $n=2$, the boundary of a smooth bounded domain $\Omega\subset M^2(K)$ is a smooth closed curve $\gamma=\partial\Omega$. Assuming $0 \in \Omega$, we denote by $\gamma_t$ the evolution of $\gamma$ under the flow \eqref{flow}, with $\Omega_t$ representing the region bounded by $\gamma_t$.

\begin{lemma}\label{2d-prop-mono}
Suppose $g :[0, \infty)\to \mathbb R$ is $C^2$ such that $g '\ge 0$ and $g ''\ge 0$, then
the quantity $$ \int_{\gamma_t} g (\Phi) \kappa d s+A_g(\Omega_t) $$ is monotone decreasing along the flow \eqref{flow}.
\end{lemma}
\begin{proof}
By results of \cite{KWWW22}, the flow \eqref{flow} preserves convexity and $0\in \Omega_t$.  
A special case of \eqref{s3.evl1} gives
\begin{equation} \label{s3.evln=2}
\frac{d}{d t} \left(\int_{\gamma_t} g(\Phi) \kappa ds+A_g(\Omega_t) \right)
=\int_{\gamma_t}\biggl(2g'(\Phi) u \kappa-g''(\Phi)|\nabla\Phi|^2\biggr)Fds,
\end{equation}
where $s$ denotes the arc-length of the curves $\gamma_t$. Letting $F=\frac{\phi'}{\kappa}-u$ in \eqref{s3.evln=2} and using
\begin{equation*}
\Delta\left(g(\Phi)\right)=g''(\Phi)|\nabla\Phi|^2+g'(\Phi)\Delta\Phi,
\end{equation*}
we have that along the flow \eqref{flow},
\begin{align*}
\frac{d}{d t}\left[\int_{\gamma_t} g (\Phi)\kappa \, d s + A_g(\Omega_t)\right]
= & \int_{\gamma_t} 2 g '(\Phi) u \left(\phi' - \kappa u\right) \, d s - \int_{\gamma_t} g''(\Phi) |\nabla\Phi|^2 \frac{\phi'}{\kappa} \, d s\\
&\quad + \int_{\gamma_t} g ''(\Phi) |\nabla \Phi|^2 u \, d s \\
\leq & \int_{\gamma_t} 2 g '(\Phi) u(\phi' - \kappa u) \, d s + \int_{\gamma_t} g ''(\Phi) |\nabla \Phi|^2 u \, d s \\
=& \int_{\gamma_t} 2 g '(\Phi) u \Delta \Phi \, d s + \int_{\gamma_t} g ''(\Phi) |\nabla \Phi|^2 u \, d s \\
=& \int_{\gamma_t} g '(\Phi) u \Delta \Phi \, d s + \int_{\gamma_t} \Delta\left(g (\Phi)\right) u \, d s \\
=& \int_{\gamma_t} g '(\Phi) u \Delta \Phi \, d s - \int_{\gamma_t} \langle \nabla\left(g (\Phi)\right), \nabla u \rangle \, d s \\
=& -\int_{\gamma_t} g ''(\Phi) u |\nabla \Phi|^2 \, d s - 2 \int_{\gamma_t} g '(\Phi) \langle \nabla u, \nabla \Phi \rangle \, d s \\
=& -\int_{\gamma_t} g ''(\Phi) u |\nabla \Phi|^2 \, d s - 2 \int_{\gamma_t} g '(\Phi) \kappa |\nabla \Phi |^2 \, d s \\
\leq & 0,
\end{align*}
where we have used $\nabla u=\kappa \nabla \Phi$ in the last equality.
\end{proof}

\begin{proof}[Proof of Theorem \ref{s1.thmn=2}]
By Lemma \ref{2d-prop-mono}, we only have to compute the limit of $\int_{\gamma_t} g(\Phi) \kappa d s+A_g(\Omega_t)$. Denote the length of the initial curve by $L$.
By results of \cite{KWWW22}, the flow \eqref{flow} is length preserving and $\gamma_t$ converges to a circle $\gamma_\infty=\partial \Omega_\infty$ centered at the origin of radius $\rho_\infty=\phi^{-1}\left(\frac{L}{2\pi}\right)$, with  $\kappa=\frac{\sqrt{4\pi^2-KL^2}}{L}$, and
\begin{equation*}
\begin{aligned}
A_g(\Omega_\infty)=
2\pi \int_0^{\rho_{\infty}}\left[g^{\prime}(\Phi(r)) \phi(r) \phi^{\prime}(r) d r+K g (\Phi(r)) \phi(r)\right] d r.
\end{aligned}
\end{equation*}
We compute
\begin{equation*}
\begin{aligned}
\int_0^{\rho_\infty} g^{\prime}(\Phi) \phi \phi^{\prime} d r
=& \int_0^{\rho_\infty}(g (\Phi(r)))^{\prime} \phi^{\prime} d r\\
=& -\int_0^{\rho_\infty} g ({\Phi}) \phi^{\prime \prime}dr+g (\Phi(\rho_\infty)) \phi^{\prime}(\rho_\infty)-g (0) \\
=& \int_0^{\rho_\infty}K g (\Phi) \phi dr+g \left(\Phi\left(\rho_{\infty}\right)\right) \phi^{\prime}\left(\rho_{\infty}\right)-g (0).
\end{aligned}
\end{equation*}
Let $G(r):=\int_0^r g(t) d t$.
By a direct calculation, $\phi^{\prime}\left(\rho_{\infty}\right)= \phi^{\prime}\left(\phi^{-1}\left(\frac{L}{2\pi}\right)\right)= \sqrt{1-\frac{KL^2}{4\pi^2}}$.
So
\begin{align*}
A_g \left[\Omega_{\infty}\right]
=& 4\pi\int_0^{\rho_{\infty}} K g (\Phi) \phi d r+ g \left(\Phi\left(\rho_{\infty}\right)\right) \sqrt{4\pi^2-KL^2}-2\pi g (0) \\
=& 4\pi K G(\Phi(\rho_\infty)) +g \left(\Phi\left(\rho_{\infty}\right)\right) \sqrt{4 \pi^2-K L^2}-2 \pi g (0).
\end{align*}
From this,
\begin{equation*}
\int_{\gamma_{\infty}} g (\Phi) \kappa d s+A_g(\Omega_\infty)
=2g \left(\Phi\left(\rho_{\infty}\right)\right) \sqrt{4 \pi^2-K L^2}-2 \pi g (0) +4 \pi K G\left(\Phi\left(\rho_{\infty}\right)\right).
\end{equation*}

Define $\textrm{A}(l)$ to be the area of the geodesic ball in $M^2(K)$ whose circumference is $l$, then
$\textrm{A}(l)=2\pi \Phi\circ \phi^{-1}\left(\frac{l}{2\pi}\right)$.
Therefore, we get the inequality
\begin{equation*}
\begin{split}
\int_\gamma g(\Phi) \kappa d s+A_g(\Omega) \ge& \int_{\gamma_\infty} g(\Phi) \kappa d s+A_g(\Omega_\infty)\\
=&2 g\left(\frac{\mathrm{A}(L)}{2 \pi}\right) \sqrt{4 \pi^2-K L^2}-2 \pi g(0)+4 \pi K G\left(\frac{\mathrm{A}(L)}{2 \pi}\right).
\end{split}
\end{equation*}
If $g$ is either strictly increasing or strictly convex, and the equality holds, then from the proof of Lemma \ref{2d-prop-mono}, we conclude that $\Phi$ is constant, and thus $\gamma$ is a geodesic circle centered at the origin.
\end{proof}

\subsection{The hyperbolic space as the ambient space}

In this section, we prove the weighted Minkowski inequality for $h$-convex hypersurfaces in the hyperbolic space $\mathbb{H}^n$. The case $n=2$ has been treated in Section \ref{sec.curve}. We only focus on $n\geq 3$ in this section.

By \eqref{s3.evl1}, along the flow \eqref{s4.Flowk=1}, the left hand side of \eqref{s1.eq-thmH} evolves by
\begin{align}\label{s4.evl-gH1}
& \frac{d}{d t} \left(\int_{\Sigma_t} g(\Phi) H d \mu_t+(n-1)A_g(\Omega_t)\right)\nonumber\\
= &2\int_{\Sigma_t}g(\Phi) \sigma_{2}\left(\frac{\phi'}{H}-\frac{u}{n-1}\right)d \mu_t+2\int_{\Sigma_t} g'u\left(\phi'-\frac{uH}{n-1}\right)d \mu_t\nonumber\\
&-\int_{\Sigma_t} g''|\nabla\Phi|^2\left(\frac{\phi'}{H}-\frac{u}{n-1}\right)d \mu_t\nonumber\\
=&I+II+III.
\end{align}
We estimate the three terms separately. Since $g\geq 0$, by the Newton-MacLaurin inequality \eqref{s2.NT2}, the first term $I$ satisfies
\begin{align*}
I\leq & \frac{1}{n-1}\int_{\Sigma_t}g(\Phi)\left((n-2)\phi'H-2u\sigma_2\right)d \mu_t.
\end{align*}
Using \eqref{s2.du} and \eqref{s2.DeltaPhi}, we can estimate the second term $II$ as following
\begin{align*}
\frac{1}{2}II=&\frac{1}{n-1}\int_{\Sigma_t}g'u\left((n-1)\phi'-uH\right)\\
=&\frac{1}{n-1}\int_{\Sigma_t}g'u\Delta\Phi\\
=&-\frac{1}{n-1}\int_{\Sigma_t}\langle \nabla_i\Phi, g''u\nabla_i\Phi+g'h_i^j\nabla_j\Phi\rangle \\
=&-\frac{1}{n-1}\int_{\Sigma_t}ug''|\nabla\Phi|^2-\frac{1}{n-1}\int_{\Sigma_t}g'h_s^ig^{sj}\nabla_j\Phi\nabla_i\Phi.
\end{align*}
Using $T_1^{ij}=\frac{\partial\sigma_2}{\partial h_{ij}}=\sigma_1g^{ij}-h_{s}^ig^{sj}$, we can rewrite the second term and obtain
\begin{align*}
\frac{1}{2}II=&-\frac{1}{n-1}\int_{\Sigma_t}ug''|\nabla\Phi|^2-\frac{1}{n-1}\int_{\Sigma_t}g'\left(\sigma_1g^{ij}-T_1^{ij}\right)\nabla_i\Phi\nabla_j\Phi\\
=&-\frac{1}{n-1}\int_{\Sigma_t}ug''|\nabla\Phi|^2-\frac{1}{n-1}\int_{\Sigma_t}g'\sigma_1|\nabla\Phi|^2+\frac{1}{n-1}\int_{\Sigma_t}g'T_1^{ij}\nabla_i\Phi\nabla_j\Phi.
\end{align*}
As $T_1$ is divergence free,
\begin{align*}
\frac{1}{n-1}\int_{\Sigma_t}g'T_1^{ij}\nabla_i\Phi\nabla_j\Phi=& \frac{1}{n-1}\int_{\Sigma_t}T_{1}^{ij}\nabla_ig\nabla_j\Phi\\
=& -\frac{1}{n-1}\int_{\Sigma_t}g\nabla_i\left(T_1^{ij}\nabla_j\Phi\right)\\
=& -\frac{1}{n-1}\int_{\Sigma_t}g\left((n-2)\phi'H-2u\sigma_2\right).
\end{align*}
Combining the above equations, we arrive at
\begin{align*}
\frac{d}{d t} \left(\int_{\Sigma_t} g(\Phi) H d \mu_t+(n-1)A_g(\Omega_t)\right)= &I+\frac{1}{2}II+\frac{1}{2}II+III\\
{\le}&-\frac{1}{n-1}\int_{\Sigma_t}ug''|\nabla\Phi|^2-{\frac{1}{n-1}\int_{\Sigma_t}g'h_s^ig^{sj}\nabla_j\Phi\nabla_i\Phi}\nonumber\\
&-\frac{1}{n-1}\int_{\Sigma_t}g'\sigma_1|\nabla\Phi|^2-\frac{1}{n-1}\int_\Sigma g''|\nabla\Phi|^2\frac{\phi'}{H}\\
\leq &0,
\end{align*}
which follows from the fact that $g'\geq 0, g''\geq 0$ and $\Sigma_t$ is $h$-convex.

On the other hand, the area of $\Sigma_t$ is fixed along the flow \eqref{s4.Flowk=1}. Applying the convergence result of the flow \eqref{s4.Flowk=1}, we obtain the inequality \eqref{s1.eq-thmH}. The equality case can be treated similary as in Theorem \ref{s1.thmRn}.

\subsection{The sphere as the ambient space}
In this section, we prove the weighted Minkowski inequality in the sphere. By the general evolution equation \eqref{s3.evl1}, along the flow \eqref{s5.Flowk=1}, we have
\begin{align*}
& \frac{d}{d t} \left(\int_{\Sigma_t} g(\Phi) H d \mu_t+(n-1)A_g(\Omega_t)\right)\\
= &2\int_{\Sigma_t}g(\Phi) \sigma_{2}\left(\frac{n-1}{H}-\frac{u}{\phi'}\right)d \mu_t+2\int_{\Sigma_t} g'(\Phi)uH\left(\frac{n-1}{H}-\frac{u}{\phi'}\right)d \mu_t\\
&-\int_{\Sigma_t} g''(\Phi)|\nabla\Phi|^2\left(\frac{n-1}{H}-\frac{u}{\phi'}\right)d \mu_t\\
=&I+II+III.
\end{align*}
We estimate the three terms separately. By the Newton-MacLaurin inequality \eqref{s2.NT2} and \eqref{s2.dT1}, the first term satisfies
\begin{align*}
I\leq & \int_{\Sigma_t}\frac{g(\Phi)}{\phi'}\left((n-2)\phi'\sigma_1-2u\sigma_2\right)d \mu_t
=~\int_{\Sigma_t}\frac{g(\Phi)}{\phi'}\nabla_i\left(T_1^{ij}\nabla_j\Phi\right).
\end{align*}
By \eqref{s2.du}, \eqref{s2.DeltaPhi} and integration by parts,
\begin{align*}
\frac{1}{2}II=&\int_{\Sigma_t}\frac{g'u}{\phi'}\left((n-1)\phi'-uH\right)\\
=&\int_{\Sigma_t}\frac{g'u}{\phi'}\Delta\Phi\\
=&-\int_{\Sigma_t}\langle \nabla_i\Phi, \frac{g'}{\phi'}h^j_i\nabla_j\Phi-\frac{ug'}{\phi'^2}\nabla_i\phi'+\frac{ug''}{\phi'}\nabla_i\Phi\rangle \\
=&-\int_{\Sigma_t}\left(\sum_{i=1}^m\frac{g'\kappa_i}{\phi'}|\nabla_i\Phi|^2+\frac{ug'}{\phi'^2}|\nabla\Phi|^2+\frac{ug''}{\phi'}|\nabla \Phi|^2\right),
\end{align*}
where we used $\phi'+\Phi=1$ in the sphere and so $\nabla_i\phi'=-\nabla_i\Phi$.

Using $T_1^{ij}=\sigma_1g^{ij}-h_{s}^ig^{sj}$ and the fact $T_1^{ij}$ is divergence free, we can rewrite the first term and obtain
\begin{align*}
\frac{1}{2}II=&-\int_{\Sigma_t}\left(\frac{g'}{\phi'}\left(\sigma_1g^{ij}-T_1^{ij}\right)\nabla_i\Phi\nabla_j\Phi+\frac{ug'}{\phi'^2}|\nabla\Phi|^2+\frac{ug''}{\phi'}|\nabla \Phi|^2\right)\\
=&-\int_{\Sigma_t}\frac{g'}{\phi'}\sigma_1|\nabla\Phi|^2+\int_{\Sigma_t}\frac{1}{\phi'}T_1^{ij}\nabla_ig\nabla_j\Phi+\int_{\Sigma_t}\left(-\frac{ug'}{\phi'^2}|\nabla\Phi|^2-\frac{ug''}{\phi'}|\nabla \Phi|^2\right)\\
=&-\int_{\Sigma_t}\frac{g'}{\phi'}\sigma_1|\nabla\Phi|^2-\int_{\Sigma_t}\frac{g}{\phi'}\nabla_i\left(T_1^{ij}\nabla_j\Phi\right)+\int_{\Sigma_t}\frac{g}{\phi'^2}T_1^{ij}{\nabla_i\phi'}\nabla_j\Phi\\
&+\int_{\Sigma_t}\left(-\frac{ug'}{\phi'^2}|\nabla\Phi|^2-\frac{ug''}{\phi'}|\nabla \Phi|^2\right)
\end{align*}
Combining the above equations, we arrive at
\begin{align*}
&\frac{d}{d t} \left(\int_{\Sigma_t} g(\Phi) H d \mu_t+(n-1)A_g(\Omega_t)\right)\\
= &I+\frac{1}{2}II+\frac{1}{2}II+III\\
\leq & -\int_{\Sigma_t}\left(\sum_{i=1}^m\frac{g'\kappa_i}{\phi'}|\nabla_i\Phi|^2+\frac{ug'}{\phi'^2}|\nabla\Phi|^2+\frac{ug''}{\phi'}|\nabla \Phi|^2\right)\\
&-\int_{\Sigma_t}\frac{g'}{\phi'}\sigma_1|\nabla\Phi|^2-\int_{\Sigma_t}\frac{g}{\phi'^2}T_1^{ij}\nabla_i\Phi\nabla_j\Phi\\
&-\int_{\Sigma_t}\frac{ug'}{\phi'^2}|\nabla\Phi|^2
-\int_{\Sigma_t} g''|\nabla\Phi|^2\frac{n-1}{H} \\
\leq & 0.
\end{align*}

On the other hand, the coarea formula implies that
\begin{align*}
\frac{d}{dt}\int_{\Omega_t}\phi'dv=&\int_{\Sigma_t} \phi'\left(\frac{n-1}{H}-\frac{u}{\phi'}\right)d\mu_t\\
=&\int_{\Sigma_t} \left(\frac{(n-1)\phi'}{H}-u\right)d\mu_t\\
\geq &0,
\end{align*}
where the inequality is due to the Heitze-Karcher inequality in $\mathbb{S}^n$ proved by Brendle \cite{Bre13}. Then Theorem \ref{s1.thmSn} follows from the monotonicity and convergence of the flow \eqref{s5.Flowk=1}.

\begin{remark}
By Gir\~{a}o and Pinheiro \cite{GP17},
\begin{equation*}
\int_{\Omega}\phi'dv\leq \frac{1}{n}\omega_{n-1}\left(\frac{|\Sigma|}{\omega_{n-1}}\right)^{\frac{n}{n-1}}
\end{equation*}
for strictly convex domain in $\mathbb{S}^n$. It would be interesting to prove  \eqref{s1.eq-thmSn} with $\int_{\Omega}\phi'dv$ replaced by $|\Sigma|$.
\end{remark}

\section{Applications of weighted inequalities to eigenvalue estimates}

In this section, we demonstrate a series of eigenvalue estimates derived by applying Theorem \ref{s1.thmn=2} and its generalization as established by the authors in \cite{KW23}. Specifically, we provide a sharp upper bound for the first positive eigenvalue of an operator involving the curvature $H_k$.

Let us now describe the problem. Let $\Sigma$ be a smooth, closed, $k$-convex hypersurface in $\mathbb{R}^n$. Consider the eigenvalue problem associated with the positive second-order differential operator $-\frac{1}{H_k} \Delta$ on $\Sigma$:
$$-\Delta f = \lambda H_k f,$$
where $H_k$ is the normalized $k$th mean curvature defined in \eqref{s2.Hk}. When $k=0$, this is just the eigenvalue problem for the classical Laplace-Beltrami operator on $\Sigma$. By classical spectral theory, the eigenvalues satisfy $0 = \lambda_0 < \lambda_1 < \lambda_2 < \cdots \to \infty$, where multiplicities are ignored. Our objective is to establish a sharp upper bound for $\lambda_1$ using weighted inequalities developed in our work. This problem can be recast in a variational framework: $\lambda_1$ is characterized as the minimum of the Dirichlet energy $\int_\Sigma |\nabla f|^2 d\mu$, subject to the constraint $\int_\Sigma H_k f d\mu = 0$ and the normalization $\int_{\Sigma} H_k f^2 d \mu=1$. Explicitly, we have
\begin{equation}\label{min}
\lambda_1 = \min \left\{ \frac{\int_\Sigma |\nabla f|^2 d\mu}{\int_\Sigma H_k f^2 d\mu} : \int_\Sigma H_k f d\mu = 0 \right\}.
\end{equation}

To tackle this eigenvalue problem, we turn our attention to Theorem \ref{s1.thmn=2}. This theorem vastly generalizes the two-dimensional result previously established with G. Wheeler and  V.-M.Wheeler \cite[Theorem 1.2]{KWWW22}. In the $\mathbb{R}^2$ setting, the result states that for a smooth convex domain $\Omega \subset \mathbb{R}^2$,
\begin{equation*}\label{2d 3term}
\int_{\partial \Omega} \kappa \frac{r^2}{2} ds \ge \frac{1}{2\pi} \left(|\partial \Omega|^2 - 2\pi |\Omega| \right),
\end{equation*}
which follows as a special case of Theorem \ref{s1.thmn=2} by choosing $g(x)=x$.
This inequality extends to higher dimensions by the authors \cite[Theorem 1.3]{KW23} in the following form:
\begin{equation}\label{ineq 3term}
\int_\Sigma H_k \frac{r^2}{2} d\mu + \frac{k}{n - k + 1} \int_\Sigma H_{k-2} d\mu \ge \frac{n+1+k}{2(n+1-k)} \omega_{n-1}^{-\frac{1}{n-k}} \left(\int_\Sigma H_{k-1} d\mu \right)^{\frac{n+1-k}{n-k}},
\end{equation}
where $\Sigma$ is a $k$-convex hypersurface in $\mathbb{R}^n$. Equality holds if and only if $\Sigma$ is a coordinate sphere. 

In our earlier work \cite{KW23}, a weighted inequality different from \eqref{ineq 3term} was employed to derive a sharp estimate for the first non-zero Steklov eigenvalue for star-shaped mean-convex domains in $\mathbb{R}^n$, generalizing the classical Weinstock inequality. Here, we demonstrate how inequality \eqref{ineq 3term} can be utilized to obtain a sharp upper bound for the first non-zero eigenvalue $\lambda_1$ of the operator $-\frac{\Delta}{H_k}$ on $k$-convex hypersurfaces in $\mathbb{R}^n$.

\begin{theorem}\label{eigen}
Let $\Sigma$ be a smooth, closed, star-shaped, and $k$-convex hypersurface in $\mathbb{R}^n$, where $n \geq 2$ and $1 \leq$ $k \leq n-1$. Then the first non-zero eigenvalue $\lambda_1$ of the operator $-\frac{1}{H_k} \Delta$ on $\Sigma$ satisfies
$$\lambda_1 \le \frac{1}{2} (n-1) |\Sigma| \left[\frac{n+1+k}{2(n+1-k)} \omega_{n-1}^{-\frac{1}{n-k}} \left(\int_\Sigma H_{k-1} d\mu \right)^{\frac{n+1-k}{n-k}} - \frac{k}{n-k+1} \int_\Sigma H_{k-2} d\mu \right]^{-1}. $$
Equality holds if and only if $\Sigma$ is a sphere.
\end{theorem}
\begin{proof}

Since $ H_k >0$, we can assume, after a suitable translation, that the weighted center of mass satisfies $\int_\Sigma H_k x d\mu = 0$, where $x = (x_1, \dots, x_n)$ is the position vector. Consequently, by \eqref{min}, for each $i = 1, \dots, n$,
$$\int_\Sigma |\nabla x_i|^2 d\mu \ge \lambda_1 \int_\Sigma H_k x_i^2 d\mu. $$
Summing over $i = 1, \dots, n$ and applying inequality \eqref{ineq 3term}, we have
\begin{equation*}
\begin{aligned}
(n-1)|\Sigma| & =\int_{\Sigma} \sum_{i=1}^n\left|\nabla x_i\right|^2 d \mu \\
& \ge\lambda_1 \int_{\Sigma} H_k r^2 d \mu \\
& \ge 2 \lambda_1\left[\frac{n+1+k}{2(n+1-k)} \omega_{n-1}^{-\frac{1}{n-k}}\left(\int_{\Sigma} H_{k-1}d\mu\right)^{\frac{n+1-k}{n-k}}-\frac{k}{n-k+1} \int_{\Sigma} H_{k-2}d\mu\right].
\end{aligned}
\end{equation*}
Rearranging yields the desired result. The case of equality follows directly from the equality condition of \eqref{ineq 3term}.
\end{proof}
In the two-dimensional case with $k=1$, this result assumes a particularly elegant geometric form. Specifically, for a smooth, closed, convex curve $\gamma$ in $\mathbb{R}^2$, the first non-zero eigenvalue $\lambda_1$ of the operator $-\frac{1}{\kappa} \Delta$ satisfies
$$
\lambda_1 \leq \frac{\pi L}{L^2-2 \pi A},
$$
where $L$ denotes the length of $\gamma$ and $A$ represents the area enclosed by $\gamma$. Equality holds if and only if $\gamma$ is a circle.


\begin{bibdiv}
\begin{biblist}
\bibliographystyle{amsplain}
\bib{AFM22}{article}{
author={Agostiniani, Virginia},
author={Fogagnolo, Mattia},
author={Mazzieri, Lorenzo},
title={Minkowski inequalities via nonlinear potential theory},
journal={Arch. Ration. Mech. Anal.},
volume={244},
date={2022},
number={1},
pages={51--85},
issn={0003-9527},
}

\bib{And94}{article}{
author={Andrews, Ben},
title={Contraction of convex hypersurfaces in Riemannian spaces},
journal={J. Differential Geom.},
volume={39},
date={1994},
number={2},
pages={407--431},
issn={0022-040X},
}

\bib{betta2008weighted}{article}{
author={Betta, Maria Francesca},
author={Brock, Friedemann},
author={Mercaldo, Anna},
author={Posteraro, Maria Rosaria},
title={Weighted isoperimetric inequalities on $\mathbb{R}^n$ and applications to rearrangements},
journal={Math. Nachr.},
volume={281},
date={2008},
number={4},
pages={466--498},
}

\bib{borell1975brunn}{article}{
author={Borell, Christer},
title={The Brunn-Minkowski inequality in Gauss space},
journal={Invent. Math.},
volume={30},
date={1975},
number={2},
pages={207--216},
}

\bib{Bre13}{article}{
author={Brendle, Simon},
title={Constant mean curvature surfaces in warped product manifolds},
journal={Publ. Math. Inst. Hautes \'{E}tudes Sci.},
volume={117},
date={2013},
pages={247--269},
issn={0073-8301},
}

\bib{BGL18}{article}{
author={Brendle, Simon},
author={Guan, Pengfei},
author={Li, Junfang},
title={An inverse curvature type hypersurface flow in space forms},
year={2018},
journal={preprint},
}

\bib{BHW16}{article}{
author={Brendle, Simon},
author={Hung, Pei-Ken},
author={Wang, Mu-Tao},
title={A Minkowski inequality for hypersurfaces in the anti--de
Sitter--Schwarzschild manifold},
journal={Comm. Pure Appl. Math.},
volume={69},
date={2016},
number={1},
pages={124--144},
issn={0010-3640},
}

\bib{cabre2013Sobolev}{article}{
author={Cabr{\'e}, Xavier},
author={Ros-Oton, Xavier},
title={Sobolev and isoperimetric inequalities with monomial weights},
journal={J. Differential Equations},
volume={255},
date={2013},
number={11},
pages={4312--4336},
}

\bib{caffarelli1984first}{article}{
author={Caffarelli, Luis},
author={Kohn, Robert},
author={Nirenberg, Louis},
title={First order interpolation inequalities with weights},
journal={Compositio Math.},
volume={53},
date={1984},
number={3},
pages={259--275},
}

\bib{CW13}{article}{
author={Chang, Sun-Yung Alice},
author={Wang, Yi},
title={Inequalities for quermassintegrals on $k$-convex domains},
journal={Adv. Math.},
volume={248},
date={2013},
pages={335--377},
}

\bib{DG16}{article}{
author={de Lima, Levi Lopes},
author={Gir\~ao, Frederico},
title={An Alexandrov-Fenchel-type inequality in hyperbolic space with an
application to a Penrose inequality},
journal={Ann. Henri Poincar\'e},
volume={17},
date={2016},
number={4},
pages={979--1002},
issn={1424-0637},
}

\bib{Ge90}{article}{
author={Gerhardt, Claus},
title={Flow of nonconvex hypersurfaces into spheres},
journal={J. Differential Geom.},
volume={32},
date={1990},
number={1},
pages={299--314},
issn={0022-040X},
}

\bib{GP17}{article}{
author={Gir\~{a}o, Frederico},
author={Pinheiro, Neilha M.},
title={An Alexandrov-Fenchel-type inequality for hypersurfaces in the
sphere},
journal={Ann. Global Anal. Geom.},
volume={52},
date={2017},
number={4},
pages={413--424},
}

\bib{GR20}{article}{
author={Gir\~{a}o, Frederico},
author={Rodrigues, Diego},
title={Weighted geometric inequalities for hypersurfaces in sub-static
manifolds},
journal={Bull. Lond. Math. Soc.},
volume={52},
date={2020},
number={1},
pages={121--136},
issn={0024-6093},
}

\bib{Gu14}{article}{
author={Guan, Pengfei},
title={Curvature measures, isoperimetric type inequalities and fully
nonlinear PDEs},
conference={
title={Fully nonlinear PDEs in real and complex geometry and optics},
},
book={
series={Lecture Notes in Math.},
volume={2087},
publisher={Springer, Cham},
},
isbn={978-3-319-00941-4},
isbn={978-3-319-00942-1},
date={2014},
pages={47--94},
}

\bib{GL09}{article}{
author={Guan, Pengfei},
author={Li, Junfang},
title={The quermassintegral inequalities for $k$-convex starshaped
domains},
journal={Adv. Math.},
volume={221},
date={2009},
number={5},
pages={1725--1732},
issn={0001-8708},
}

\bib{GL15}{article}{
author={Guan, Pengfei},
author={Li, Junfang},
title={A mean curvature type flow in space forms},
journal={Int. Math. Res. Not. IMRN},
date={2015},
number={13},
pages={4716--4740},
issn={1073-7928},
}

\bib{GL21}{article}{
author={Guan, Pengfei},
author={Li, Junfang},
title={Isoperimetric type inequalities and hypersurface flows},
journal={J. Math. Study},
volume={54},
date={2021},
number={1},
pages={56--80},
issn={2096-9856},
}

\bib{HLW22}{article}{
author={Hu, Yingxiang},
author={Li, Haizhong},
author={Wei, Yong},
title={Locally constrained curvature flows and geometric inequalities in
hyperbolic space},
journal={Math. Ann.},
volume={382},
date={2022},
number={3-4},
pages={1425--1474},
issn={0025-5831},
}

\bib{HS99}{article}{
author={Huisken, Gerhard},
author={Sinestrari, Carlo},
title={Convexity estimates for mean curvature flow and singularities of
mean convex surfaces},
journal={Acta Math.},
volume={183},
date={1999},
number={1},
pages={45--70},
issn={0001-5962},
}

\bib{KM14}{article}{
author={Kwong, Kwok-Kun},
author={Miao, Pengzi},
title={A new monotone quantity along the inverse mean curvature flow in
$\Bbb{R}^n$},
journal={Pacific J. Math.},
volume={267},
date={2014},
number={2},
pages={417--422},
issn={0030-8730},
}

\bib{KM15}{article}{
author={Kwong, Kwok-Kun},
author={Miao, Pengzi},
title={Monotone quantities involving a weighted $\sigma_k$ integral along
inverse curvature flows},
journal={Commun. Contemp. Math.},
volume={17},
date={2015},
number={5},
pages={1550014, 10},
issn={0219-1997},
}

\bib{KW23}{article}{
author={Kwong, Kwok-Kun},
author={Wei, Yong},
title={Geometric inequalities involving three quantities in warped
product manifolds},
journal={Adv. Math.},
volume={430},
date={2023},
pages={Paper No. 109213, 28},
issn={0001-8708},
}

\bib{KWWW22}{article}{
author={Kwong, Kwok-Kun},
author={Wei, Yong},
author={Wheeler, Glen},
author={Wheeler, Valentina-Mira},
title={On an inverse curvature flow in two-dimensional space forms},
journal={Math. Ann.},
volume={384},
date={2022},
number={1-2},
pages={285--308},
issn={0025-5831},
}

\bib{Mc05}{article}{
author={McCoy, James A.},
title={Mixed volume preserving curvature flows},
journal={Calc. Var. Partial Differential Equations},
volume={24},
date={2005},
number={2},
pages={131--154},
issn={0944-2669},
}

\bib{Q15}{article}{
author={Qiu, Guohuan},
title={A family of higher-order isoperimetric inequalities},
journal={Commun. Contemp. Math.},
volume={17},
date={2015},
number={3},
pages={1450015, 20},
issn={0219-1997},
}

\bib{Re73}{article}{
author={Reilly, Robert C.},
title={Variational properties of functions of the mean curvatures for
hypersurfaces in space forms},
journal={J. Differential Geometry},
volume={8},
date={1973},
pages={465--477},
issn={0022-040X},
}

\bib{Ros87}{article}{
author={Ros, Antonio},
title={Compact hypersurfaces with constant higher-order mean curvatures},
journal={Rev. Math. Iber.},
volume={3},
year={1987},
pages={447--453},
}

\bib{SX19}{article}{
author={Scheuer, Julian},
author={Xia, Chao},
title={Locally constrained inverse curvature flows},
journal={Trans. Amer. Math. Soc.},
volume={372},
date={2019},
number={10},
pages={6771--6803},
}

\bib{Sch14}{book}{
author={Schneider, Rolf},
title={Convex bodies: the Brunn-Minkowski theory},
series={Encyclopedia of Mathematics and its Applications},
volume={151},
edition={expanded edition},
publisher={Cambridge University Press, Cambridge},
date={2014},
pages={xxii+736},
isbn={978-1-107-60101-7},
}

\bib{Ur90}{article}{
author={Urbas, John I. E.},
title={On the expansion of starshaped hypersurfaces by symmetric
functions of their principal curvatures},
journal={Math. Z.},
volume={205},
date={1990},
number={3},
pages={355--372},
issn={0025-5874},
}

\bib{WZ23}{article}{
author={Wei, Yong},
author={Zhou, Tailong},
title={New weighted geometric inequalities for hypersurfaces in space
forms},
journal={Bull. Lond. Math. Soc.},
volume={55},
date={2023},
number={1},
pages={263--281},
}

\bib{WU24}{article}{
author={Wu, Jie},
title={Weighted Alexandrov-Fenchel type inequalities for hypersurfaces in
$\Bbb R^n$},
journal={Bull. Lond. Math. Soc.},
volume={56},
date={2024},
number={8},
pages={2634--2646},
}

\end{biblist}
\end{bibdiv}
\end{document}